\documentclass[12pt, reqno]{amsart}
\usepackage{hyperref}
\usepackage[alphabetic]{amsrefs}
\usepackage{amsmath}
\usepackage{pgfplots}
\pgfplotsset{compat=1.15}
\usepackage{mathrsfs}
\usepackage{amsaddr}
\usepackage{graphicx}
\usepackage{pdfpages}
\usepackage{comment}
\usetikzlibrary{arrows}
\allowdisplaybreaks

\usepackage{tikz-cd}
 \usetikzlibrary{calc}

\theoremstyle{plain}
\newtheorem{thm}{Theorem}[section]
\newtheorem{defn}[thm]{Definition}
\newtheorem{rem}[thm]{Remark}

\newtheorem{lem}[thm]{Lemma}
\newtheorem{nota}[thm]{Notation}
\numberwithin{equation}{section}

\foreach \x in {1,2,3,4}{\expandafter\xdef\csname bar\x\endcsname{\noexpand\bar{\x}}}

\begin{document}

\title{Quantum Symmetry on Potts Model}
\begin{abstract}
	We formulate the notion of quantum group symmetry of the hamiltonian corresponding to Potts model and compute it for few simple models. Our examples illustrate how a slight change of the model parameter may result in a drastic change of the quantum symmetry group, (in some cases, the classical symmetry group remains unaffected) signifying a case of phase transition.
\end{abstract}

\author{Debashish Goswami}

\address{Statistics and Mathematics Unit, Indian Statistical Institute\\ 203, B.T. Road, Kolkata 700108, Indi\\\textnormal{email: \texttt{debashish\_goswami@yahoo.co.in}}}

\author{Sk Asfaq Hossain}
\address{Statistics and Mathematics Unit, Indian Statistical Institute\\ 203, B.T. Road, Kolkata 700108, India\\\textnormal{email: \texttt{asfaq1994@gmail.com}}}

\maketitle

\section{Introduction}
Symmetry is one of the oldest and most important central concepts used in nearly every branch of science. Classically, symmetry has been modelled by group actions but a more general paradigm of Hopf algebras or quantum groups is also quite common by now (see, e.g. (\cite{drinfeld1986icm}), (\cite{MR901157}), (\cite{MR1358358}) and  references therein, just to mention a few). In this article, we want to consider quantum group symmetry of a very popular class of physical model, namely the Potts model.  Typically, Potts models are studied on infinite lattices or infinite graphs and thermodynamic properties are analysed. However, we will consider only finite graphs for a simpler mathematical treatment. There is another important departure from the conventional framework. Instead of looking at phase transition in terms of break of continuity or smoothness of some suitable thermodynamic term, we'll take a symmetry viewpoint, following the classical ideas of Landau (\cite{landau1969statistical}), which says that a change of the group of symmetry of the underlying physical system signifies a change of phase. For example, gaseous phase has lot more symmetry than liquid which is constrained to have a fixed volume. Similarly, liquid has more symmetry than solid. We will consider quantum group symmetry generalising group symmetry. We give a few examples where a slight perturbation  of the hamiltonian of the model leads to a significant change of the (quantum) symmetry group, which may be interpreted as a phase change. It is perhaps interesting to note that in several models of condensed matter physics, 
 there are theoretical and experimental explanations of effect of doping-induced phase transition in terms of change of the point symmetry group of the underlying crystal structures (see for instance, (\cite{perez2012symmetry})). We propose extension of such ideas to the realm of quantum group symmetry. However, we must admit that our mathematical examples  are in not the realistic models taken from physics or any other natural science. All we want in this article is to illustrate our novel perspective  through a few kind of `toy-models' to encourage  an in-depth analysis of real-life Potts model based on quantum group symmetry. The simplicity of the models in our paper help us compute the exact quantum symmetry of them, using which we could show instance of phase transition with the change of symmetry. 
 \par 
 It may be noted that the interplay of quantum groups and operator algebras with the models (including Potts model) of statistical mechanics goes back to the seminal work of Jones, (see for example, \cite{MR990215}) leading to a theory of work connecting subfactor theory, quantum field theory and so on with such physical models. Later Banica showed (see for example, \cite{MR1654119}) that the spin and vertex models of Jones do come from quantum groups.
 \par
 However, an important difference between our approach to (quantum) symmetry from the works mentioned above is that we consider a (quantum) symmetry group (co)acting on the vertices of the lattice or graph commuting with the hamiltonian instead of an on-site symmetry coming from the permutation of the set of states or the commutator with the site-to-site transfer matrices as in (\cite{MR1317365}). Nevertheless, as the transfer matrices are closely related to the hamiltonian of the underlying model, there may be an interesting connection between the two approaches, which we will try to explore later.
 \par

Let us now give a brief sketch of the paper. We consider Potts model on a finite graph with a finite state space, identified with elements of a suitable cyclic group and formulate a notion of (co)action of compact quantum groups on the space of functions on the vertex set of the graph such that the coaction commutes with the hamiltonian (i.e. it gives a symmetry of the underlying model). Following the line of (\cite{MR2174219}) we prove the existence of a universal compact quantum group which gives symmetry of the Potts model in the above sense, calling it the quantum group of symmetry of the model. Then we compute this quantum group in a few examples and show how a slight change of the parameters of the model can result in a rather remarkable change of quantum symmetry. More interestingly, we give an example where classical symmetry group remains the same even after the change of the model parameters but quantum symmetry group changes. This makes a strong case for studying quantum group symmetry of physical models.

\section{Prelimineries}
\subsection{Compact quantum group}
We start with a brief description of compact quantum group and few related concepts associated with it.   For more details, see (\cite{MR1645264}), (\cite{MR1358358}), (\cite{MR901157}) and (\cite{MR3559897}). All C* algebras are assumed to be unital in this paper and the tensor product considered is the minimal tensor product among C* algebras.
\begin{defn}
	A compact quantum group is a pair $(\mathcal{S},\Delta)$  where $\mathcal{S}$ is a unital C* algebra and $\Delta:\mathcal{S}\rightarrow \mathcal{S}\otimes \mathcal{S}$ is a homomorphism of C* algebras satisfying the following conditions:
	\begin{enumerate}
		\item $(\Delta\otimes id)\Delta=(id\otimes\Delta)\Delta$ (Coassociativity).
		\item Each of the linear spans of $\Delta(\mathcal{S})(1\otimes \mathcal{S})$ and $\Delta(\mathcal{S})(\mathcal{S}\otimes 1)$ is norm-dense in $\mathcal{S}\otimes \mathcal{S}$.
	\end{enumerate}
\end{defn}

It is known that there exists a unique Haar state on a compact quantum group which is the non-commutative analogue of Haar measure on a classical compact group.
\begin{defn}
	The Haar state $h$  on a compact quantum group is the unique state on $\mathcal{S}$ which satisfies the following conditions:
	\begin{equation*}
	(h\otimes id)\Delta(a)=h(a)1_{S} \quad \text{and}\quad (id\otimes h)\Delta(a)=h(a)1_{S}
	\end{equation*}
for all $a\in \mathcal{S}$.
\end{defn}
\begin{defn}
	A quantum group homomorphism $\Phi$ between two compact quantum groups $(\mathcal{S}_1,\Delta_1)$ and $(\mathcal{S}_2,\Delta_2)$ is a C* algebra homomorphism $\Phi:\mathcal{S}_1\rightarrow\mathcal{S}_2$ satisfying the following condition:
	\begin{equation*}
		(\Phi\otimes\Phi)\circ\Delta_1=\Delta_2\circ\Phi
	\end{equation*}
\end{defn}
Now we breifly state few facts from the co-representation theory of compact quantum groups. 
\begin{defn}
	Let $n\in\mathbb{N}$. An n dimensional co-representation of a compact quantum group $(\mathcal{S},\Delta)$ is an $\mathcal{S}$ valued $n\times n$ matrix $(v_{ij})_{n\times n}$ of satisfying the following property:
	\begin{equation*}
		\Delta(v_{ij})=\sum_{k=1}^n v_{ik} \otimes v _{kj}
	\end{equation*}
$v_{ij}$'s are called matrix elements of the n dimensional co-representation of $(\mathcal{S},\Delta)$.
\end{defn}
It is known that for a compact quantum group $(\mathcal{S},\Delta)$, there is a dense subalgebra $\mathcal{S}_0$ generated by the matrix elements of finite dimensional co-representations of $(\mathcal{S}_0,\Delta)$ which forms a Hopf * algebra in its own right. $h$ is faithful when restricted to $\mathcal{S}_0$ i.e. for $a$ in $\mathcal{S}_0$, $a=0$ whenever $h(a^*a)=0$. 
\subsection{Co-actions and quantum automorphism groups } 
\begin{defn}
	Let $\mathcal{A}$ be a unital C* algebra. A co-action of a compact quantum group $(\mathcal{S},\Delta)$ on $\mathcal{A}$ is a * homomorphism $\mathcal{A}\rightarrow \mathcal{A}\otimes\mathcal{S}$ satisfying the following conditions:
	\begin{enumerate}
		\item $(\alpha\otimes id)\alpha=(id\otimes\Delta)\alpha$.
		\item Linear span of $\alpha(\mathcal{A})(1_{\mathcal{A}}\otimes\mathcal{S})$ is norm-dense in $\mathcal{A}\otimes\mathcal{S}$.
	\end{enumerate}
\end{defn}
\begin{defn}
	A co-action of $(\mathcal{S},\Delta)$ on $\mathcal{A}$ is said to be faithful if there does not exist any proper Woronowicz C* subalgebra $\mathcal{S}_1$ of $\mathcal{S}$ such that $\alpha$ is a coaction of $\mathcal{S}_1$ on $\mathcal{A}$. A continuous linear functional $\tau$ on $\mathcal{A}$ is said to be invariant under $\alpha$ if the following holds:
	\begin{equation*}
		(\tau\otimes id)\alpha(a)=\tau(a)1_{\mathcal{S}}
	\end{equation*}
	for all $a$ in $\mathcal{S}$.
\end{defn}
In this paper we will only be considering faithful co-actions.
\par

For a unital C*algebra $\mathcal{A}$ the $\textbf{category of quatum transformation groups}$ is the category whose objects are the compact quantum groups co-acting on $\mathcal{A}$ and morphisms are quantum group homomorphisms intertwining  such co-actions, that is, for any morphism $\Phi$ between two compact quantum transformation groups $(\mathcal{S},\Delta)$ and $(\mathcal{S}',\Delta')$ the following diagram commutes:

\begin{center}
	\includegraphics[]{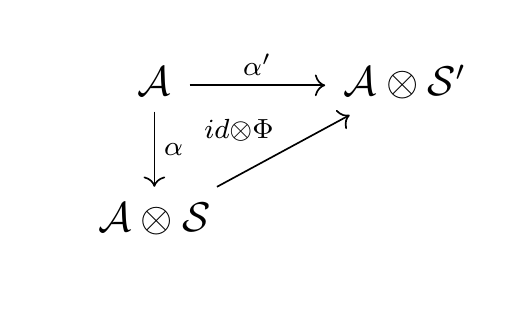}
\end{center}

where $\alpha$ and $\alpha'$ are co-actions of $(\mathcal{S},\Delta)$ and $(\mathcal{S}',\Delta')$ on $\mathcal{A}$ respectively. The universal object in this category, if it exists, is said to be the $\textbf{quantum automorphism}$ $\textbf{group}$ of $\mathcal{A}$. 

\begin{rem}
	For an arbitrary unital C* algebra the quantum automorphism group might not exist. The algebra of $n\times n$ complex matrices is such an example. However the quantum automorphism group does exist if we restrict the category to a smaller category where co-actions are trace preserving. For details, see (\cite{MR1637425}).
\end{rem}
\subsection{Quantum permutation group on a finite set}
Let $X_n=\{1,2,...,n\}$. Consider the algebra $C(X_n)$. We note that
\begin{equation*}
C(X_n):=span\{\chi_i, i=1,2,...,n\}. 
\end{equation*}
where $\chi_i$ is the characteristic function on $i$, that is, $\chi_i(k)=1$ if $k=i$ and $\chi_i(k)=0$ otherwise.
  The universal object in the $\textbf{category of quantum transformation groups}$ for the algebra $C(X_n)$ exists and Wang described it in (\cite{MR1637425}). It is called the quantum permutation group on $X_n$ and is denoted by $(S_n^+,\Delta_n)$. $S_n^+$ is the universal C* algebra generated by generators $\{u_{ij},i,j=1,2,...,n\}$ satisfying the following relations: 

\begin{enumerate}
	\item $u_{ij}^2=u_{ij}=u_{ij}^*$ for $i,j=1,2,...,n$
	\item $\sum_{i=1}^n u_{ij}=\sum_{j=1}^n u_{ij}=1$
\end{enumerate}

The co-product $\Delta:{{S_n}^+}\rightarrow{{S_n}^+}\otimes{{S_n}^+}$ is given by $\Delta(u_{ij})=\sum_{k=1}^n u_{ik}\otimes u_{kj}$
     The cannonical co-action $\alpha$ of $(S_n^+,\Delta)$ on $C(X_n)$ is given by
     \begin{equation}
     \alpha(\chi_{i})=\sum_{j=1}^n \chi_j \otimes u_{ji}; \quad i=1,2,..,n
     \end{equation}
     \par
    It is also worth mentioning that $(S_n^+,\Delta_n)$ is a compact quantum group of Kac type (see (\cite{MR1637425}) and references within) which makes the corrsponding Haar state $h:S_n^+\rightarrow\mathbb{C}$ tracial, that is, $h(ab)=h(ba)$ for all $a,b\in S_n^+$.      
     \begin{nota}
    	let $\alpha$ be a co-action of a compact quantum group $(S,\Delta)$ on $C(X_n)$. The co-representation matrix of the co-action $\alpha$   is an $S$-valued $n\times n$ matrix $Q=(q_{ij})_{n\times n}$ such that the following holds:
    \begin{equation*}
    	\alpha(\chi_i)=\sum_{j=1}^n \chi_j \otimes q_{ji}, \quad i=1,2,..,n
    \end{equation*}
As our co-actions are faithful, $S$ is generated by $q_{ij}$'s as a C* algebra.
       
    \end{nota}  
\subsection{Quantum automorphism group of a finite graph}
 let $(V,E)$ be an undirected finite graph with vertex set $V$ and edge set $E$. We also assume that $(V,E)$ has no loops and no multiple edges among any two points. Let $A=(a_{ij})_{n\times n}$ denote the adjacency matrix of $(V,E)$, that is, $a_{ij}=1$ if $(i,j)\in E$ and $a_{ij}=0$ otherwise.
 
 \begin{defn}
    A co-action $\alpha$ of $(S,\Delta)$ on $C(V)$ is said to be co-acting by preserving the quantum symmetry of $(V,E)$ if $Q$ commutes with $A$ where $Q$ is the co-representation matrix of $\alpha$. 
 \end{defn} 
Consider the category whose objects are compact quantum groups co-acting on $(V,E)$ by preserving the quantum symmetry of the graph $(V,E)$ and morphisms are quantum group morphisms intertwining two such co-actions. The following result is due to Banica in 2005. See (\cite{MR2146039}).
\begin{thm}\label{uni.obj.finite.graph}
	The universal object in the above mentioned category exists and is given by $S_n^+/UA-AU$ where $U$ denotes the corepresentation matrix of the cannonical co-action of $S_n^+$ on $C(V)$. It is called the quantum automorphism group of $(V,E)$. 
\end{thm}
In fact, the adjacency matrix $A$ does not have any special role in the proof of existence of the universal object in (\cite{MR2146039}) and by arguments similar to those used in (\cite{MR2146039}) and (\cite{MR2174219}) one can obtain a version of  theorem (\ref{uni.obj.finite.graph}) for any complex valued matrix $A=(a_{ij})_{n\times n}$ indexed by the vertices. This, in particular covers the case of weighted graphs and finite metric spaces. For the convenience of the reader, let us state this as a theorem:
\begin{thm}\label{uni.obj.complex.matrix}
	Let $X_n$ be a finite set with $n$ elements and $A\in M_n(\mathbb{C})$. Let us consider the category whose objects are compact quantum groups co-acting on $C(X_n)$ in such a way that the corepresentation matrix $Q$ commutes with $A$. Then the universal object in this category exists and is given by $S_n^+/UA-AU$ where $U$ is the co-representation matrix of the cannonical co-action of $(S,\Delta)$ on $C(X_n)$.
\end{thm}

\section{Potts Model and its quantum symmetry}
 Let us say a few words about the basics of statistical mechanics for the readers without any physics background. The fundamental idea of statistical mechanics is to consider an "ensemble" or totality of all possible states of a physical system in equilibrium and assign a probability distribution, usually the so-called Boltzmann (or Gibbs) distribution. The probability distribution (or the probability density) of a state depends on the corresponding energy level (typically given by the hamiltonian, say $H$) and the absolute temparature of the system ($T$). Usually, the probability $P(\rho)$ of a state $\rho$ is taken as a multiple of $e^{-\beta H(\rho)}$ where $\beta=1/T$. Then one considers the average expected value of the states calculated according to the above distribution. The partition function is given by 
 \begin{equation*}
 Z=\sum_{\rho}e^{-\beta H(\rho)}
  \end{equation*}
  where $\rho$ varies over all states. This is of central importance and its behavior with respect to the inverse temparature $\beta$ is studied. Any break of analyticity in $\beta$ is thought as a sign of a phase transition. However, as remarked earlier we have taken a different approach to phase transition in this paper. 
 \par
 Potts model, or more general vertex and spin models are some of the most popular and useful models arising primarily in statistical (including quantum statistical) mechanics, but they have found wide applications in many other areas of physics and even other scientific (including social sciences) disciplines. Usually, the physical picture of a Potts model considers atoms occupying the vertices of a lattice or more general graphs, each of which can be in one of a specified set of physical states. The edges joining two such atoms are thought of as bonds between them and an atom interacts only with the nearest neighbours.
 \par 
 We will be using a simpler version of Potts model for simpler  mathematical treatment. For more details on Potts model, see (\cite{martin1991}).
\par
Let $(V,E)$ be an undirected finite graph with no loops and multiple edges among two vertices. A q-state Potts model ($q\in\mathbb{N}$ and $q\geq 2$) on $(V,E)$ consists of a set of configurations $\Omega_P$ and a hamiltonian $H_P:\Omega_P\rightarrow\mathbb{C}$ defined as follows:  

\begin{defn}
A configuration $\omega$ for a q-state Potts Model on $(V,E)$ is a function from $V$ to the set $X_q$. The hamiltonian $H_P$ is defined to be: 
\begin{equation}\label{1}
H_P(\omega) := \sum_{(i,j)\in E} J_{ij}\delta_{\omega(i),\omega(j)} \qquad \text{for all}\quad \omega \in \Omega_P 
\end{equation}
where $J_{ij}\in\mathbb{C}$ and $J_{ij}=J_{ji}$ for all $i,j\in V$. The expression $\delta_{\omega(i),\omega(j)}$ is equal to $1$ if $\omega(i)=\omega(j)$ and $0$ otherwise.
\end{defn}
By taking $J'_{ij}=J_{ij}a_{ij}$ we get,
\begin{equation*}
	H_P(\omega)=\sum_{i,j\in V} J'_{ij}\delta_{\omega(i),\omega(j)}  \qquad \text{for all}\quad \omega \in \Omega_P 
\end{equation*}
Now we discuss the notion of quantum symmetry on Potts model.
\par 

Let $\alpha$ be a co-action of a compact quantum group $(\mathcal{S},\Delta)$ on $C(V)$. We want to describe what it means for $\alpha$ to preserve the hamiltonian $H_P$. Such a co-action can be described to preserve the quantum symmetry of the q-state Potts model on $(V,E)$. For our purpose, it is convenient to see a configuration $\omega$ as an element of $C(V) \otimes C^*(\mathbb{Z}_q)$ such that, 
\begin{equation*}
\omega(i)=\chi_{g_i}\qquad \text{for some}\quad g_i \in \mathbb{Z}_q.
\end{equation*}
\par
Let $\tau:C^*(\mathbb{Z}_q) \rightarrow \mathbb{C}$ be a linear functional defined by $\tau(f)=f(e)$, where $e$ is the identity of the cyclic group $\mathbb{Z}_q$. 
 Let us define a bilinear form $<,>_{H_P}$ on $C(V)\otimes C^*(\mathbb{Z}_q)$ by 
$$<f,h>_{H_P} = \sum_{i,j \in V} J'_{ij} \tau ({f(i)}^**h(j)) $$
where $f$ and $h$ are arbitrary elements in $C(V) \otimes C^*(\mathbb{Z}_q)$  and ${f(i)}^*(g)=\overline{f(i)(g^{-1})}$.
We Observe that,

\begin{align*}
<f,h>_{H_P}&=\sum_{i,j \in V} J'_{ij} \tau (({f(i)}^**h(j))\\
&=\sum_{i,j\in V}J'_{ij}\tau\big((\sum_{g_1\in \mathbb{Z}_q}\overline{f(i)(g_1)}\chi_{g_1}^*)*(\sum_{g_2\in \mathbb{Z}_q}h(j)(g_2)\chi_{g_2})\big)\\
&=\sum_{\substack{i,j\in V\\g_1,g_2 \in \mathbb{Z}_q}}J'_{ij}\overline{f(i)(g_1)}h(j)(g_2)\tau\big(\chi_{g_1^{-1}g_2}\big)\\
&=\sum_{\substack{i,j\in V\\g\in \mathbb{Z}_q}}J'_{ij}\overline{f(i)(g)}h(j)(g)\qquad (\text{as $\tau(\chi_g)=1$ iff $g=e$})\\
\end{align*}
 
Let $\omega \in \Omega_P$. We observe that,
\begin{align}
\notag<\omega,\omega>_{H_P}&=\sum_{\substack{i,j\in V\\g\in \mathbb{Z}_q}}J'_{ij}\overline{\omega(i)(g)}\omega(j)(g)\\\notag
&=\sum_{i,j\in V}J'_{ij}(\sum_{g\in\mathbb{Z}_q}\overline{\omega(i)(g)}\omega(j)(g))\\\notag
&=\sum_{i,j\in V}J'_{ij}\delta_{g_i,g_j}\qquad (\text{as $\omega(i)(g)=\chi_{g_i}$})\\\label{1.9}
&=H_P(\omega).
\end{align}

$<,>_{H_P}$ induces a $S$ valued bilinear form $\widetilde{<,>_{H_P}}$ on $C(V)\otimes C^*(\mathbb{Z}_q)\otimes S$ given by
\begin{equation*}
\widetilde{<f\otimes a,h\otimes b>_{H_P}} := <f,h>_{H_P} a^*b
\end{equation*}
\par
Let $\alpha'$ be the  coaction on $C(V)\otimes C^*(\mathbb{Z}_q)$ induced by $\alpha$. It is given by $\alpha'=(id\otimes \sigma_{23})(\alpha \otimes id)$ where $\sigma_{23}$ is the standard flip between 2nd and 3rd coordinates.
\par
\begin{defn}\label{a.0}
 $\alpha$  is said to preserve the hamiltonian of q-state Potts model on $(V,E)$ if the following holds:
\begin{equation}\label{2.00}
<\omega,\omega>_{H_P}=\widetilde{<\alpha'(\omega),\alpha'(\omega)>_{H_P}} 
\end{equation}
\end{defn} 
\begin{nota}\label{b} 
	 We introduce some notations for our convenience. Let $\beta,\gamma\in V$.
	\begin{equation*}
	Q^{\beta \gamma} := \sum_{k,l \in V} J'_{kl} q_{k\beta } q_{l\gamma }.
	\end{equation*}
Let $f\in C(V)$ be defined by, 
	\begin{equation*}
	f(\beta)=\sum_{j \in V} J'_{\beta j} = \sum_{i \in V} J'_{i \beta} \hspace{5mm} \forall \beta \in V.
	\end{equation*}  
\end{nota} 
\par
By evaluationg right hand side of equation $\eqref{2.00}$ we get,

\begin{align}
\widetilde{<\alpha'(\omega), \alpha'(\omega)>_{H_P}}&=<\sum_{i \in V}\alpha'(\chi_i\otimes\omega(i)),\sum_{j \in V}\alpha'(\chi_j\otimes\omega(j))>_{H_P}\notag\\
&=\sum_{i,j,k,l\in V}<\chi_k \otimes\omega(i), \chi_l\otimes \omega(j)>_{H_P} q_{ki}q_{lj}\notag\\
&=\sum_{\substack{i,j,k,l\in V\\g\in \mathbb{Z}_q}}J'_{kl}\overline{\omega(i)(g)}\omega(j)(g)q_{ki}q_{lj}\notag\\
&=\sum_{\substack{i,j\in V\\g\in \mathbb{Z}_q}}\overline{\omega(i)(g)}\omega(j)(g) Q^{ij} \label{2.01}
\end{align}

\begin{rem} \label{c}
	 From ($\ref{1.9}$) and ($\ref{2.01}$) it follows that  $\alpha$ preserves the hamiltonian of q-state Potts model on $(V,E)$ iff the following holds:
	\begin{equation}\label{2.1}
	\sum_{\substack{i,j\in V\\g\in \mathbb{Z}_q}}\overline{\omega(i)(g)}\omega(j)(g)J'_{ij}1=\sum_{\substack{i,j\in V\\g\in \mathbb{Z}_q}}\overline{\omega(i)(g)}\omega(j)(g)Q^{ij} \qquad \text{for all}\quad \omega \in \Omega_P 
	\end{equation}
\end{rem}

\begin{lem}\label{d}
 If the co-representation matrix $Q$ corresponding to $\alpha$ commutes with $J'$, then $\alpha$ preserves the hamiltonian $H_P$ in the sense of definition ($\ref{a.0}$).
\end{lem}

\begin{proof}
 Let $i,j\in V$. We observe that, 
\begin{equation*}
\begin{aligned}
Q^{ij}=\sum_{k,l \in V} J'_{kl} q_{ki} q_{lj} &= \sum_{k\in V} q_{ki} (\sum_{l\in V}J'_{kl}  q_{lj}) \\ 
							   &=\sum_{k\in V} \sum_{l\in V} J'_{lj} q_{ki} q_{kl} \\
							   &=\sum_{k\in V} J'_{ij} q_{ki}=J'_{ij}1\\
\end{aligned}
\end{equation*}

Hence from ($\ref{2.1}$), the result follows.	
\end{proof}
\par
It is interesting to ask, whether the converse is true, that is, the corepresentation matrix $Q$ of a hamiltonian preserving coaction $\alpha$ commutes with $J'$ or not. It turns out to be true. To show that, we will need two following lemmas.

\begin{lem}\label{e}
Let $\alpha$ be the co-action of $(S,\Delta)$ on $C(V)$ as before and  $h$ be the haar functional on $(S, \Delta)$. If $h(Q^{\beta\gamma})=J'_{\beta\gamma}$ for all $\beta,\gamma \in V$, then $QJ'=J'Q$, where $Q$ is the co-representation matrix of $\alpha$.
\end{lem}

\begin{proof}
Let $\beta, \gamma \in V$. Then 

\begin{align}
\Delta(Q^{\beta\gamma}) &=\Delta(\sum_{k,l \in V}J'_{kl}q_{k\beta }q_{l\gamma })\notag \\
&=\sum_{k,l \in V}J'_{kl}\Delta(q_{k\beta })\Delta(q_{l\gamma })\notag \\
&=\sum_{k,l \in V}J'_{kl}(\sum_{k'\in V}q_{k k'}\otimes q_{k'\beta})(\sum_{l'\in V}q_{ll'}\otimes q_{l'\gamma}) \notag \\
&=\sum_{k,l,k',l' \in V}J'_{kl}(q_{ kk'}q_{ll'} \otimes q_{k'\beta}q_{l'\gamma}) \label{2.2}
\end{align}

As $h$ is the Haar functional, we have,
\begin{equation}
(h \otimes Id)\Delta(s)=h(s)1 \hspace{5mm} \forall s \in S
\end{equation} 
From equation$\eqref{2.2}$ we get,
\begin{align*}
(h \otimes Id)\Delta(Q^{\beta\gamma}) &=\sum_{k,l,k',l' \in V} J'_{kl}h(q_{kk'}q_{ll'})q_{k'\beta }q_{l'\gamma} \\
&=\sum_{k',l'\in V}(h(\sum_{k,l\in V}J'_{kl}q_{kk'}q_{ll'}))q_{k'\beta }q_{l'\gamma} \\
&=\sum_{k',l'\in V}q_{k'\beta}q_{l'\gamma }h(Q^{k'l'}) \\
&=\sum_{k',l'\in V}q_{k'\beta}q_{l'\gamma}J'_{k'l'}=Q^{\beta \gamma}
\end{align*}
From our hypothesis it follows that 
\begin{equation}\label{2.3}
Q^{\beta\gamma}=J'_{\beta\gamma}1.
\end{equation}

Finally we observe,
\begin{align*}
(QJ')_{ij}=\sum_{k\in V}q_{ik}J'_{kj}
&=\sum_{k\in V}q_{ik}(\sum_{k',l\in V}J'_{k'l}q_{k'k}q_{lj}) \quad \text{(from $\eqref{2.3}$)} \\
&=\sum_{k,l\in V}J'_{il}q_{ik}q_{lj} \\
&=\sum_{l\in V}J'_{il}q_{lj}=(J'Q)_{ij} 
\end{align*}

Hence we get that $Q$ and $J'$ commutes.
\end{proof}
From lemma ($\ref{d}$) and lemma ($\ref{e}$), we get the following result:
\begin{thm}\label{f}
Let $(S,\Delta)$ be a compact quantum group co-acting on $C(V)$. The co-representation matrix $Q$ commutes with $J'$ if and only if $Q^{\beta\gamma}=J'_{\beta\gamma}1$ for all $\beta, \gamma \in V$. 
\end{thm}

\par

\begin{lem}\label{g}
Let $\alpha$ be a coaction of $(S,\Delta)$ on $C(V)$. If $\alpha$ preserves the hamiltonian for a q-state Potts Model on $(V,E)$, then $\alpha (f)=f\otimes 1$ where $f$ is described in notation($\ref{b}$) 
\end{lem}

\begin{proof}
 To show that $\alpha(f) = f\otimes 1$, it is enough to show

\begin{equation}\label{2.31}
 \sum_{i\in V}f(i)q_{\beta i}= f(\beta)1 
\end{equation}
for all $\beta\in V$.
\par
Let us fix $\beta$ in $V$ and $g_0 \in \mathbb{Z}_q$ such that $g_{0} \neq e$. We define $\omega:V \rightarrow C^*(\mathbb{Z}_q)$ by $\omega(\beta)=\chi_{g_0}$ and $\omega(i) = \chi_{e}$ for $i\neq \beta$.
For $\omega$, we evaluate right hand side of equation $\eqref{2.1}$ as follows:
\begin{align*}
	\sum_{\substack{i,j\in V\\g\in \mathbb{Z}_q}}\overline{\omega(i)(g)}\omega(j)(g)Q^{ij}&=\sum_{\substack{i\in V\\g\in \mathbb{Z}_q}}\overline{\omega(i)(g)}\omega(\beta)(g)Q^{i\beta}+\sum_{\substack{j\in V\\g\in \mathbb{Z}_q}}\overline{\omega(\beta)(g)}\omega(j)(g)Q^{\beta j}\\&\:\:\:\:\:+\sum_{\substack{i,j\neq \beta\\g\in \mathbb{Z}_q}}\overline{\omega(i)(g)}\omega(j)(g)Q^{ij}\\
	&=\sum_{i,j\neq \beta}Q^{ij}\\
	&=\sum_{i,j\neq \beta}\sum_{k,l\in V}J'_{kl}q_{ki}q_{lj}\\
	&= \sum_{k,l\in V} J'_{kl}(1 - q_{k\beta })(1-q_{l\beta }) \\
	&=\sum_{k,l\in V}J'_{kl}(1-q_{k\beta }-q_{l\beta }) \\
	&=\sum_{k\in V} f(k)-2\sum_{l\in V} f(l)q_{l\beta }.
	\end{align*}

For $\omega$, evaluating left hand side of equation$\eqref{2.1}$ we get,
\begin{align*}
\sum_{\substack{i,j\in V\\g\in \mathbb{Z}_q}}J'_{ij}\overline{\omega(i)(g)}\omega(j)(g)&=\sum_{\substack{i\in V\\g\in \mathbb{Z}_q}}J'_{i\beta}\overline{\omega(i)(g)}\omega(\beta)(g)+\sum_{\substack{j\in V\\g\in \mathbb{Z}_q}}J'_{\beta j}\overline{\omega(\beta)(g)}\omega(j)(g)\\
&\:\:\:\: +\sum_{\substack{i,j\neq\beta\\g\in \mathbb{Z}_q}}J'_{ij}\overline{\omega(i)(g)}\omega(j)(g)\\
&=\sum_{i,j \neq \beta}J'_{ij}\\
&=\sum_{j\neq\beta}(f(j)-J'_{\beta j})\\
&=\sum_{j\in V}f(j)-2f(\beta)\\
\end{align*}
By our hypothesis and remark ($\ref{c}$) we know that  equation (\ref{2.1}) holds. Hence, we conclude that (\ref{2.31}) is true.
\end{proof}
Now we are in a position to prove the main result of this paper.

\begin{thm}\label{h}
	Let $\alpha$ is a coaction of $(S,\Delta)$ on $C(V)$. If $\alpha$ preserves the hamiltonian of q-state Potts model on $(V,E)$, then the co-representation matrix $Q$ of $\alpha$ commutes with $J'$.
\end{thm}

\begin{proof}
 We fix $\beta,\gamma \in V$ such that $\beta\neq\gamma$ and $g_0\in\mathbb{Z}_q$ which is not $e$. We define a configuration $\omega:V \rightarrow C^*(\mathbb{Z}_q)$ by
	\begin{equation*}
		\begin{aligned}
		\omega(\beta)&=\chi_{g_0},\\ \omega(\gamma)&=\chi_{g_0},\\  \omega(i)&=\chi_e\quad \text{for}\quad i\neq\beta,\gamma.\\
	\end{aligned}
	 \end{equation*}.
	\par
	Evaluating right hand side of equation$\eqref{2.1}$ for $\omega$ we get,
	
	\
	\begin{align}
	\notag\sum_{\substack{i,j\in V\\g\in \mathbb{Z}_q}}\overline{\omega(i)(g)}\omega(j)(g)Q^{ij}=&\sum_{\substack{j\neq\beta,\gamma\\g\in\mathbb{Z}_q}}\overline{\omega(\beta)(g)}\omega(j)(g)Q^{\beta j}+\sum_{\substack{j\neq\beta,\gamma\\g\in\mathbb{Z}_q}}\overline{\omega(\gamma)(g)}\omega(j)(g)Q^{\gamma j}\\\notag
	&+\sum_{\substack{i\neq\beta,\gamma\\g\in\mathbb{Z}_q}}\overline{\omega(i)(g)}\omega(\beta)(g)Q^{i\beta}+\sum_{\substack{i\neq\beta,\gamma\\g\in\mathbb{Z}_q}}\overline{\omega(i)(g)}\omega(\gamma)(g)Q^{i\gamma}\\\notag
	&+\sum_{\substack{i\neq\beta,\gamma\\j\neq\beta,\gamma\\ g\in\mathbb{Z}_q}}\overline{\omega(i)(g)}\omega(j)(g)Q^{ij}+Q^{\beta\gamma}+Q^{\gamma\beta}\\\notag
	=&\sum_{\substack{i\neq\beta,\gamma\\j\neq\beta,\gamma}}Q^{ij}+Q^{\beta\gamma}+Q^{\gamma\beta}\\\notag
	=&\sum_{k,l\in V}J'_{kl}(\sum_{\substack{i\neq\beta,\gamma\\j\neq\beta,\gamma}}q_{ki}q_{lj})+Q^{\beta\gamma}+Q^{\gamma\beta}\\\notag
	=&\sum_{k,l\in V}J'_{kl}(1-q_{k\beta}-q_{k\gamma })(1-q_{l\beta}-q_{l\gamma })+Q^{\beta\gamma}+Q^{\gamma\beta}\\\notag
	=&\sum_{k,l \in V}J'_{kl}(1-q_{k\beta}-q_{k\gamma }-q_{l\beta}+q_{k\gamma }q_{l\beta }-q_{l\gamma }+q_{k\beta}q_{l\gamma })\\\notag&\:\:+Q^{\beta\gamma}+Q^{\gamma\beta}\\\notag
	=&\sum_{k\in V} f(k)-f(\beta)1-f(\gamma)1-f(\beta)1+Q^{\gamma\beta}-f(\gamma)1\\\notag&\:\:+Q^{\beta\gamma}+Q^{\beta\gamma}+Q^{\gamma\beta}\\\label{2.4}
	=&\sum_{k\in V}f(k)-2f(\beta)1-2f(\gamma)1+2Q^{\beta\gamma}+2Q^{\gamma\beta}
	\end{align}
    
	By evaluating left hand side of equation (\ref{2.1}) for $\omega$  we get,
	\begin{align}
	\notag\sum_{\substack{i,j\in V\\g\in \mathbb{Z}_q}}\overline{\omega(i)(g)}\omega(j)(g)J'_{ij}=&\sum_{\substack{j\neq\beta,\gamma\\g\in\mathbb{Z}_q}}\overline{\omega(\beta)(g)}\omega(j)J'_{\beta j}+\sum_{\substack{j\neq\beta,\gamma\\g\in\mathbb{Z}_q}}\overline{\omega(\gamma)(g)}\omega(j)(g)J'_{\gamma j}\\\notag
	&+\sum_{\substack{i\neq\beta,\gamma\\g\in\mathbb{Z}_q}}\overline{\omega(i)(g)}\omega(\beta)(g)J'_{i\beta}+\sum_{\substack{i\neq\beta,\gamma\\g\in\mathbb{Z}_q}}\overline{\omega(i)(g)}\omega(\gamma)(g)J'_{i\gamma}\\\notag
	&+\sum_{\substack{i\neq\beta,\gamma\\j\neq\beta,\gamma\\\notag g\in\mathbb{Z}_q}}\overline{\omega(i)(g)}\omega(j)(g)J'_{ij}+J'_{\beta\gamma}+J'_{\gamma\beta}\\\notag
	=&\sum_{\substack{i\neq\beta,\gamma\\j\neq\beta,\gamma}}J'_{ij}+J'_{\beta\gamma}+J'_{\gamma\beta}\\\notag
	=&\sum_{i\neq\beta,\gamma}(f(i)-J'_{i\beta}-J'_{i\gamma})+J'_{\beta\gamma}+J'_{\gamma\beta}\\\notag
	=&\sum_{i\neq\beta,\gamma}f(i)-\sum_{i\neq\beta,\gamma}J'_{i\beta}-\sum_{i\neq\beta,\gamma}J'_{i\gamma}+J'_{\beta\gamma}+J'_{\gamma\beta}\\\notag
	=&\sum_{i\in V}f(i)-f(\beta)-f(\gamma)-f(\beta)+J'_{\gamma\beta}-f(\gamma)+J'_{\beta\gamma}\\\notag&\:\:+J'_{\beta\gamma}+J'_{\gamma\beta}\\\label{2.41}
	=&\sum_{i\in V}f(i)-2f(\beta)-2f(\gamma)+2J'_{\beta\gamma}+2J'_{\gamma\beta}
	\end{align}

	From our hypothesis and remark (\ref{c}) we know that equation ($\ref{2.1}$) holds. Hence, from $\eqref{2.4}$ and $\eqref{2.41}$ and lemma (\ref{g}) we get,
	\begin{align*}
	Q^{\beta\gamma}+Q^{\gamma\beta}&=(J'_{\beta\gamma}+J'_{\gamma\beta})1\\
	\text{which implies},\qquad h(Q^{\beta\gamma})&=J'_{\beta\gamma}1.\qquad (\text{as $h$ is tracial on $S$})
	\end{align*}
	Since our choice of $\beta$, $\gamma$ was arbitrary, from Lemma($\ref{e}$) the theorem follows.
	
\end{proof}

\par
\begin{lem}\label{h.1}
	let $\alpha$ be a co-action of a compact quantum group $(S,\Delta)$ on $C(V)$ which preserves the hamiltonian of a q-state Potts model on $(V,E)$. Let $k\in\mathbb{C}$ and $K:=span \{\chi_v\in C(V)| f(v)=k \}$. Then $\alpha(K)\subseteq K\otimes S$. 
\end{lem}
\begin{proof}
	From lemma (\ref{g}), it follows that $\alpha(f)=f\otimes 1$ where $f$ is described in notation (\ref{b}). Hence we have,
	\begin{equation*}\label{x}
	\sum_{i\in V}f(i)q_{ji}=f(j)1 \qquad \text{for all}\quad j\in V.
	\end{equation*} 
	Now we choose $v\in V$ such that $f(v)=k$ and $w\in V$ such that $f(w)\neq k$. To prove our claim it is enough to show that $q_{vw}=0$.
	We observe that,
	\begin{equation*}
	q_{vw}f(v) = q_{vw}\sum_{k\in V}f(k)q_{vk}=q_{vw}f(w)\\
	\end{equation*}
	As $f(v)\neq f(w)$, it follows that $q_{vw}=0$.
\end{proof}
\begin{thm}\label{i}
 Let $(S,\Delta)$ be a compact quantum group which  co-acts on $C(V)$ via $\alpha$ by preserving the hamiltonian of the q-state Potts Model on $(V,E)$. We define $\alpha^{(2)}$ as
\begin{equation*}
	\alpha ^{(2)}= (id\otimes id\otimes m)(id\otimes\sigma_{23}\otimes id)(\alpha\otimes\alpha)
	\end{equation*} 
	where $m$ is the multiplication on $S$ and $\sigma_{23}$ is the standard flip on second and third coordinate. Let $K_c:= Span\{\chi_k \otimes \chi_l| J'_{kl}=c\}$. Then $\alpha^{(2)}(K_c) \subseteq K_c \otimes S$.  
\end{thm}
\begin{proof}
	Let $i,j \in V$. We have
	\begin{equation*}
	 \alpha^{(2)}(\chi_i \otimes \chi_j)= \sum_{k,l \in V}\chi_k\otimes \chi_l \otimes q_{ki}q_{lj}  
	\end{equation*}
	Our claim is equivalent to the statement that 
	  \begin{equation}\label{3}
	  q_{ki}q_{lj}=0 \quad\text{if}\quad J'_{kl}\neq J'_{ij}.
	  \end{equation}
	\par   
	 As $\alpha$ preserves the hamiltonian on $(V,E)$, from theorem($\ref{h}$) and remark(${\ref{f}}$) it follows that, \begin{equation*}
	 Q^{ij}=J'_{ij}1 \qquad \text{for all $i,j$ in $V$}.
	 \end{equation*}
	 \par
	 Let us first assume that $J'$ to be a real matrix. Let $\phi : V \times V \rightarrow \mathbb{R}$ be defined by $\phi(i,j)=J'_{ij}$. Let $Image(\phi)=\{k_1,k_2,...,k_r\}$ where $k_1<k_2<...<k_r \in \mathbb{R}$. Let $(i_0,j_0)\in V\times V$ such that $J'_{i_0,j_0}=k_r$. We note that,
	 \begin{equation}\label{4}
	 	\sum_{k,l \in V}(J'_{i_0j_0}-J'_{kl})h(q_{ki_0}q_{lj_0})
	 	=h(J'_{i_0j_0}1-Q^{i_0j_0})=0
	 \end{equation}
	 where $h$ is the Haar state on $(S,\Delta)$.We Observe that
	 \begin{equation}\label{5}
	 	h(q_{ki}q_{lj})=h(q_{ki}q_{lj}q_{ki}) \geq 0 \qquad \text{for all}\quad i,j,k,l \in V. 
	 \end{equation}
	 As $(J'_{i_0j_0}-J'_{kl})\geq 0$, from equation (\ref{4})  and equation (\ref{5})  it follows that, for any pair $(k,l)\in V \times V$ with $J'_{i_0j_0}\neq J'_{kl} $, $h(q_{ki_0}q_{lj_0})=0$. As $h$ is faithful on the underlying dense Hopf * Algebra of $(S,\Delta)$, it follows that $q_{ki_0}q_{lj_0}=0$ for any pair $(k,l)\in V \times V$ with $J'_{i_0j_0}\neq J'_{kl} $. 
	 \par
	 We have proved equation (\ref{3}) to be true for a pair $(i,j)\in V\times V$ with $J'_{ij}=k_r$. For other $k_n$'s we proceed by induction. We fix $k_n\in Image(\Phi)$ and assume that,
	 \begin{equation*}
	 	q_{ki}q_{lj}=0 \quad \text{where}\quad J'_{ij}=k_n\quad\text{and}\quad J'_{ij}\neq J'_{kl}.
	 \end{equation*}. We want to see whether the above statement holds true for $k_{n-1}$. We choose $i,j\in V$ such that $J'_{ij}=k_{n-1}$. For $k,l\in V$, if $J'_{kl}>J'_{ij}$, then $q_{ki}q_{lj}=0$ by our induction hypothesis. As before, we observe that,
	 \begin{equation}\label{6}
	 	\sum_{\substack{k,l\in V\\\text{such that} J'_{kl}\leq J'_{ij}}}(J'_{ij}-J'_{kl})h(q_{ki}q_{lj})=\sum_{k,l \in V}(J'_{ij}-J'_{kl})h(q_{ki}q_{lj})
	 	=h(J'_{ij}1-Q^{ij})=0
	 \end{equation} 
	 From (\ref{6}), it follows that,  for $k,l\in V$ with $J'_{kl}< J'_{ij} $, $h(q_{ki}q_{lj})=0$, which in turn implies that $q_{ki}q_{lj}=0$. The theorem is proved when the matrix $J'$ is real.
	 \par
	 Now let us assume that $J'=J'_R + iJ'_I$ where $J'_R$ and $J'_I$ are real matrices. As the co-representation matrix $Q$ consists of positive elements in $S$, it follows that $Q$ commutes with $J'$ iff $Q$ commutes with both $J'_R$ and $J'_I$. Using the real case we conclude that for $i,j,k,l \in V$,
	 \begin{align*}
	 	q_{ki}q_{lj}&=0 \qquad \text{if}\quad J'_{R_{ij}}\neq J'_{R_{kl}}\\
	 	q_{ki}q_{lj}&=0 \qquad \text{if}\quad J'_{I_{ij}}\neq J'_{I_{kl}}
	 \end{align*}
	 As $J'_{ij}=J'_{R_{ij}}+iJ'_{I_{ij}}$, it follows that $q_{ki}q_{lj}=0$ if $J'_{ij}\neq J'_{kl}$ for $i,j,k,l\in V$.
	 \end{proof}
 \begin{lem}\label{j}
 	Let $\alpha$ be a coaction of $S$ on $C(V)$ preserving the hamiltonian $H_P$ on $(V,E)$. It is immidiate that $a_{ij}=0$ implies $J'_{ij}=0$. If we further assume that $J'_{ij}=0$ implies $a_{ij}=0$, then $\alpha$ also preserves the quantum symmetry of the underlying graph $(V,E)$, that is, the co-representation matrix $Q$ commutes with the adjacency matrix $A$ of $(V,E)$.  
 \end{lem}
\begin{proof}
	We define $R^{ij}=\sum_{k,l\in V}a_{kl}q_{ki}q_{lj}$. From theorem ($\ref{f}$), we know that $Q$ commutes with $A$ iff $R^{ij}=a_{ij}1$ for all $i,j\in V$. To prove the result, it is enough to show
	\begin{equation*}
		R^{ij}=a_{ij}1\quad\text{for all}\quad i,j\in V.
	\end{equation*} From our hypothesis and theorem ($\ref{i}$) it follows that
	\begin{equation}\label{6.1}
		q_{ki}q_{lj}=0 \qquad \text{if $a_{ij}=1$ but  $a_{kl}=0$}
	\end{equation}
	Let $i,j\in V$ be such that $a_{ij}=1$. We observe that, 
	\begin{equation*}
	 1=\sum_{k,l\in V}q_{ki}q_{lj}=\sum_{\substack{k,l\in V\\\text{with} \hspace{1mm} a_{kl}=1}}q_{ki}q_{lj}+\sum_{\substack{k,l\in V\\\text{with}\hspace{1mm} a_{kl}=0}}q_{ki}q_{lj}=R^{ij}+0\quad \text{(from $\eqref{6.1}$)}
	\end{equation*} 
	The other case where $a_{ij}=0$ similarly follows.

\end{proof}

\begin{thm}
 There exists a unique universal object in the category of compact quantum groups co-acting on $(V,E)$ by preserving the hamiltonian of q-state Potts model. 
\end{thm}
\begin{proof}
	From lemma (\ref{d}) and theorem (\ref{h}), it follows that a compact quantum group $(S,\Delta)$ co-acts on $(V,E)$ via preserving the hamiltonian iff the co-representation matrix $Q$ commutes with $J'$. Hence, from theorem (\ref{uni.obj.complex.matrix}) our claim follows.
\end{proof}

Let us call this unique universal object  quantum symmetry group of Potts model on $(V,E)$. Furthermore, if the hypothesis in lemma ($\ref{j}$) is satisfied, then this object turns out to be a quantum subgroup of the quantum automorphsim group of the underlying graph $(V,E)$. .
\section{Examples}

\subsection{Example 1} 
Here we look at an example of Potts model where we observe that a slight fluctuation in hamiltonian can destroy the quantum symmetry present in system and turn it into a classical one.  This change in quantum symmetry can indicate towards a kind of phase transition in the system.
\par
 We start with the graph of a cube as shown in figure (\ref{ex1.}). Let the set of vertices be $V=\{1,2,3,4,1',2',3',4'\}$ and the edges $E$ are as shown in the picture.

 Let $\lambda \in \mathbb{C}$. Consider the hamiltonian $H_{\lambda}$ given by,
\begin{equation}
	H_{\lambda}(\omega)=\sum_{i,j\in V}J'_{ij}\delta_{\omega(i),\omega(j)} \quad \forall \omega\in\Omega_P
\end{equation}
where $J'_{ij}$ is given by,
\begin{equation*}
\begin{aligned}
J'_{ij}&=1\quad\text{if}\quad (i,j)\neq (4,4'),(4',4).\\
J'_{44'}&=J'_{4'4}=\lambda\\
\end{aligned}
\end{equation*}
\begin{figure}
	\begin{center}
		\includegraphics[]{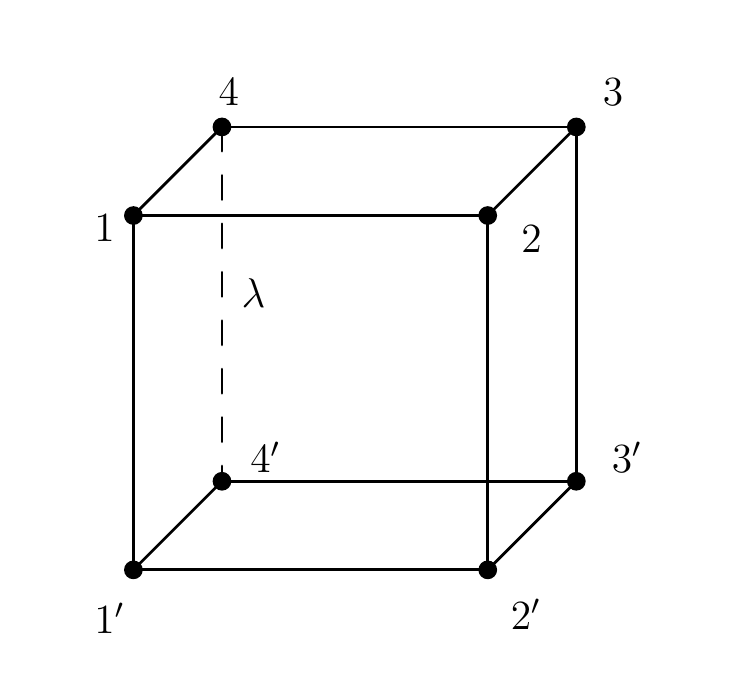}
		\end{center}
	\caption{The graph of a cube with a specified hamiltonian on it. }\label{ex1.}
\end{figure}
\begin{rem}
 When $\lambda=1$, The quantum symmetry group of Potts Model is the quantum automoprhism group for the graph $(V,E)$ which is a non-classical compact quantum group. See for instance (\cite{MR2336835}).
\end{rem}
  
\begin{lem} When $\lambda$ is not $1$, the quantum symmetry group of Potts model on $(V,E)$ becomes commutative.
\end{lem}
\begin{proof}Let $\lambda$ be a complex number which is not $1$. Let $\alpha$ be a coaction of a compact quantum group $(S,\Delta)$ on $C(V)$ preserving the hamiltonian $H_{\lambda}$. From theorem($\ref{h}$), the co-representation matrix $Q$ of $\alpha$ commutes with $J'$. 
\par
 From lemma ($\ref{h.1}$) it follows that $q_{i4}=q_{4j}=q_{i4'}=q_{4'j}=0$ when $i\neq 4,4'$ and $j\neq 4,4'$. Hence it follows that $$q_{44'}=q_{4'4} \quad\text{and}\quad q_{44}=q_{4'4'}$$
 \par
Using $QJ'=J'Q$, we get the following commutation relations:
\begin{equation}\label{y}
\begin{aligned}
q_{44}=&(J'Q)_{14}=(QJ')_{14}=q_{11}+q_{13}\\
0=&(J'Q)_{24}=(QJ')_{24}=q_{21}+q_{23}\\
q_{4'4}=&(J'Q)_{1'4}=(QJ')_{1'4}=q_{1'1}+q_{1'3}\\
0=&(J'Q)_{2'4}=(QJ')_{2'4}=q_{2'1}+q_{2'3}\\
q_{4'4'}=&(J'Q)_{1'4'}=(QJ')_{1'4'}=q_{1'1'}+q_{1'3'}\\
0=&(J'Q)_{2'4'}=(QJ')_{2'4'}=q_{2'1'}+q_{2'3'}\\
\end{aligned}
\end{equation}
From equations (\ref{y}) We note that $q_{21}+q_{23}=0$, which implies $q_{21}=0$ and $q_{23}=0$. Similarly $q_{2'1}=q_{2'3}=q_{21'}=q_{23'}=q_{2'1'}=q_{2'3'}=0$. The co-representation matrix $Q$ becomes,
$$
\begin{bmatrix}
q_{11}&0&q_{13}&0&q_{11'}&0&q_{13'}&0\\
0&q_{22}&0&0&0&q_{22'}&0&0\\
q_{31}&0&q_{33}&0&q_{31'}&0&q_{33'}&0\\
0&0&0&q_{44}&0&0&0&q_{44'}\\
q_{1'1}&0&q_{1'3}&0&q_{1'1'}&0&q_{1'3'}&0\\
0&q_{2'2}&0&0&0&q_{2'2'}&0&0\\
q_{3'1}&0&q_{3'3}&0&q_{3'1'}&0&q_{3'3'}&0\\
0&0&0&q_{4'4}&0&0&0&q_{4'4'}\\
\end{bmatrix}
$$
By equating 1st row of $(J'Q)$ and $(QJ')$ we get,
\begin{equation}\label{z}
	\begin{aligned}
	q_{1'1}=(J'Q)_{11}&=(QJ')_{11}=q_{11'}\\
	q_{22}=(J'Q)_{12}&=(QJ')_{12}=q_{11}+q_{13}=q_{44}\\
	q_{1'3}=(J'Q)_{13}&=(QJ')_{13}=q_{13'}\\
	q_{1'1'}=(J'Q)_{11'}&=(QJ')_{11'}=q_{11}\\
	q_{22'}=(J'Q)_{12'}&=(QJ')_{12'}=q_{11'}+q_{13'}=q_{44'}\\
    q_{1'3'}=(J'Q)_{13'}&=(QJ')_{13'}=q_{13}\\
   \end{aligned}
\end{equation}
Finally, from (\ref{z}) and (\ref{y}) we observe that,
\begin{equation}
	q_{i'j}=q_{ij'}\quad \text{and}\quad q_{i'j'}=q_{ij}\quad \text{for all}\quad i,j\in \{1,2,3,4\}.
\end{equation}
These are enough relations to conclude that the entries of $Q$ commute with each other. Hence the quantum symmetry group of Potts model on $(V,E)$ is commutative.
\end{proof}
\subsection{Example 2}
We look at an example of Potts model where slight fluctuation of hamiltonian changes the quantum symmetry of the system but does not affect its classical symmetry. 
\par 
We consider the graph $(V,E)$ shown in figure (\ref{ex2.}). Let $\lambda_1, \lambda_2\in \mathbb{C}$ be such that $|\lambda_1|\leq 1$ and $|\lambda_2|\leq 1$. Consider the following hamiltonian $H_{\lambda_1,\lambda_2}$ on $(V,E)$ given by
\begin{equation*}
H_{\lambda_1,\lambda_2}(\omega)=\sum_{i,j\in V}J'_{ij}\delta_{\omega(i),\omega(j)} \quad \forall\quad \omega\in\Omega_P
\end{equation*}

	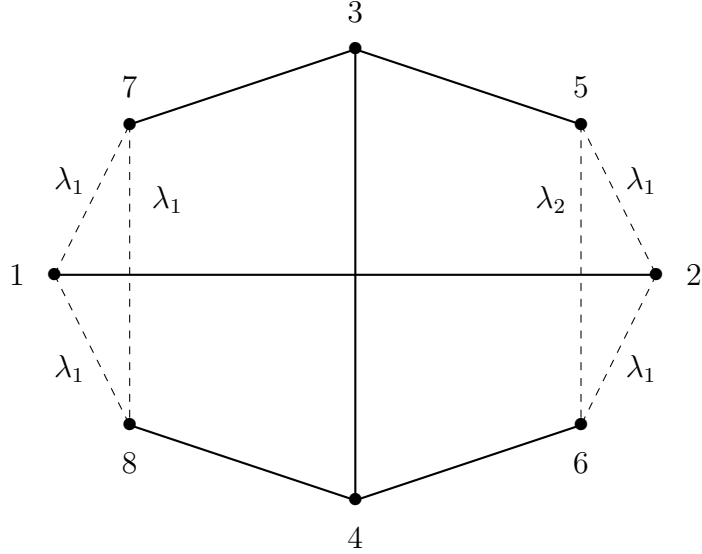
\begin{figure}
	\begin{center}
		\begin{tikzpicture}
		\node at (-4,0)	{$\bullet$};
		\node at (-4.5,0) {$1$};
		\node at (4,0) {$\bullet$};
		\node at (4.5,0) {$2$};
		\node at (0,3) {$\bullet$};
		\node at (0,3.5) {$3$};
		\node at (0,-3) {$\bullet$};
		\node at (0,-3.5) {$4$};
		\node at (-3,2) {$\bullet$};
		\node at (-3,2.5) {$7$};
		\node at (3,2) {$\bullet$};
		\node at (3,2.5) {$5$};
		\node at (-3,-2) {$\bullet$};
		\node at (-3,-2.5) {$8$};
		\node at (3,-2) {$\bullet$};
		\node at (3,-2.5) {$6$};
		\node at (-3.8,1.25) {$\lambda_1$};
		\node at (-3.8,-1.25) {$\lambda_1$};
		\node at (3.8,1.25) {$\lambda_1$};
		\node at (3.8, -1.25) {$\lambda_1$};
		\node at (-2.5,1) {$\lambda_1$};
		\node at (2.6,1) {$\lambda_2$};
		\draw[dashed] (-4,0) -- (-3,2);
		\draw[dashed] (-4,0) -- (-3,-2);
		\draw[thick] (-4,0) -- (4,0);
		\draw[dashed] (-3,2) -- (-3,-2);
		\draw[thick] (-3,2) -- (0,3);
		\draw[thick] (0,3) -- (3,2);
		\draw[thick] (0,3) -- (0,-3);
		\draw[dashed] (3,2) -- (3,-2);
		\draw[dashed] (3,2) -- (4,0);
		\draw[dashed] (4,0) -- (3,-2);
		\draw[thick] (3,-2) -- (0,-3);
		\draw[thick] (0,-3) -- (-3,-2);
		\end{tikzpicture}
\end{center}
\caption{The graph in example 2 with the corresponding hamiltonian.}\label{ex2.}
\end{figure}
where $(J'_{ij})_{i,j\in V}$ is a symmetric matrix and $J'_{78}=J_{17}=J'_{18}=J'_{25}=J'_{26}=\lambda_1, J'_{56}=\lambda_2$ and $J'_{ij}=1$ otherwise. The underlying graph $(V,E)$ does not have any quantum symmetry (see section (4.4) of \cite{schmidt2020thesis}). In light of theorem (\ref{i}), we observe that,
\begin{rem}
	When $\lambda_1\neq 0$ and $\lambda_1=\lambda_2$,  the quantum symmetry group of Potts model is $\mathbb{C}(\mathbb{Z}_2)\otimes\mathbb{C}(\mathbb{Z}_2)$.
\end{rem} 

\begin{lem}
	When $\lambda_1=0$ and $\lambda_1\neq\lambda_2$, the quantum symmetry group of Potts model on $(V,E)$ becomes non-classical.
\end{lem} 
\begin{proof}
	let $(S,\Delta)$ be the quantum symmetry group for Potts model on $(V,E)$ co-acting on $C(V)$ via $\alpha$. As before, from theorem ($\ref{h}$) it follows that the co-representation matrix $Q$ commutes with $J'$.
	\par
Note that from lemma (\ref{h.1}) it follows that 
\begin{equation}\label{e.2.1}
	\begin{aligned}
	q_{3j}=q_{4j}&=0\quad \text{for}\quad j\neq 3,4.\\
	q_{1j}=q_{2j}&=0\quad \text{for}\quad j\in\{3,4,5,6\}.\\
	q_{7j}=q_{8j}&=0 \quad \text{for} \quad j\in\{3,4,5,6\}.
\end{aligned}
\end{equation}
We observe that,
\begin{equation}\label{e.2.2}
\begin{aligned}
	q_{34}=(J'Q)_{74}&=(QJ')_{74}=q_{78}\\
	q_{34}=(J'Q)_{54}&=(QJ')_{54}=q_{56}\\
	0=(J'Q)_{13}&=(QJ')_{13}=q_{17}\\
	0=(J'Q)_{14}&=(QJ')_{14}=q_{18}\\
	0=(J'Q)_{23}&=(QJ')_{23}=q_{27}\\
	0=(J'Q)_{24}&=(QJ')_{24}=q_{28}\\
\end{aligned}
\end{equation}
In light of equations (\ref{e.2.1}) and (\ref{e.2.2}), the co-representation matrix $Q$ becomes,
$$
\begin{bmatrix}
	1-p&p&0&0&0&0&0&0\\
	p&1-p&0&0&0&0&0&0\\
	0&0&1-q&q&0&0&0&0\\
	0&0&q&1-q&0&0&0&0\\
	0&0&0&0&1-q&q&0&0\\
	0&0&0&0&q&1-q&0&0\\
	0&0&0&0&0&0&1-q&q\\
	0&0&0&0&0&0&q&1-q
\end{bmatrix}
$$
where two projections $p$ and $q$ do not need to commute. Hence we conclude that the quantum symmetry group of Potts model on $(V,E)$ is $\mathbb{C}(\mathbb{Z}_2)*\mathbb{C}(\mathbb{Z}_2)$.
\end{proof}  
\begin{rem}
	For $\lambda_1=0$ and $\lambda_1\neq\lambda_2$, the quantum symmetry group for Potts model is $\mathbb{C}(\mathbb{Z}_2)*\mathbb{C}(\mathbb{Z}_2)$ which is indeed a non-classical comapct quantum group. On the other hand, the classical symmetry group for Potts model is $\mathbb{C}(\mathbb{Z}_2)\otimes\mathbb{C}(\mathbb{Z}_2)$, which is same as the case when $\lambda_1\neq 0$ and $\lambda_1=\lambda_2$. Hence we observe that, slight changes in the parameters $\lambda_1$ and $\lambda_2$ keep the classical symmetry same but affect quantum symmetry in the system drastically. 
\end{rem}

\textbf{Acknowledgement}:
\begin{itemize}
\item We thank the referee for his/her insightful comments and poinitng out some more references related to our work.
\item It is also to be noted that Debashish Goswami is partially supported by J.C. Bose national fellowship awarded by D.S.T., Government of India.
\end{itemize}

\par

                                                                                                                                     \end{document}